\documentclass{article}
\usepackage{graphicx} 
\usepackage{amsmath,amsthm,amssymb,amsfonts,eucal,fullpage,latexsym,mathrsfs,stmaryrd}
\usepackage{comment,authblk}
\usepackage[normalem]{ulem}
\usepackage[T1]{fontenc}
\usepackage{hyperref}

\usepackage[dvipsnames]{xcolor}
\title{Metric projections, zeros of optimal polynomial approximants, and some extremal problems in Hardy spaces}

\author[1]{Catherine B\'en\'eteau}
\author[2]{Raymond Cheng}
\author[1]{Christopher Felder}
\author[1]{Dmitry Khavinson}
\author[3]{Myrto Manolaki}
\author[4]{Konstantinos Maronikolakis}

\affil[1]{Department of Mathematics and Statistics, University of South Florida, \href{mailto:cbenetea@usf.edu}{cbenetea@usf.edu}, \href{mailto:felderc@usf.edu}{felderc@usf.edu}, \href{mailto:dkhavins@usf.edu}{dkhavins@usf.edu}}
\affil[2]{Department of Mathematics and Statistics, Old Dominion University, \href{mailto:rcheng@odu.edu}{rcheng@odu.edu}}
\affil[3]{School of Mathematics and Statistics, University College Dublin, \href{mailto:myrto.manolaki@ucd.ie}{myrto.manolaki@ucd.ie}}
\affil[4]{Department of Mathematics, Bilkent University, \href{mailto:conmaron@gmail.com}{conmaron@gmail.com}}

\date{\today}

\usepackage{enumerate}
\usepackage[dvipsnames]{xcolor}
\usepackage{lipsum}
\usepackage{array}
\usepackage{lineno}
\theoremstyle{plain}
\newtheorem{Theorem}{Theorem}
\numberwithin{Theorem}{section}
\numberwithin{equation}{section}
\theoremstyle{plain}
\newtheorem{Proposition}[Theorem]{Proposition}
\newtheorem{Corollary}[Theorem]{Corollary}
\newtheorem*{Corollary*}{Corollary}

\newtheorem*{Theorem*}{Theorem}
\newtheorem{Lemma}[Theorem]{Lemma}
\theoremstyle{definition}

\newtheorem{Conjecture}[Theorem]{Conjecture}
\newtheorem{Example}[Theorem]{Example}
\newtheorem{Remark}[Theorem]{Remark}

\def\phi{\varphi}

\renewcommand{\leq}{\leqslant}
\renewcommand{\geq}{\geqslant}

\newcommand{\D}{\mathbb{D}}

\newcommand{\T}{\mathbb{T}}
\newcommand{\N}{\mathbb{N}}

\newcommand{\chris}[1]{{\color{Plum}#1}}

\begin{document}

\maketitle

\begin{abstract}
The well-known proof of Beurling's Theorem in the Hardy space $H^2$, which describes all shift-invariant subspaces, rests on calculating the orthogonal projection of the unit constant function onto the  subspace in question. Extensions to other Hardy spaces $H^p$ for $0 < p < \infty$ are usually obtained by reduction to the $H^2$ case via inner-outer factorization of $H^p$ functions. In this paper, we instead explicitly calculate the metric projection of the unit constant function onto a shift-invariant subspace of the Hardy space $H^p$ when $1<p<\infty$. This problem is equivalent to finding the best approximation in $H^p$ of the conjugate of an inner function. 
In $H^2$, this approximation is always a constant, but in $H^p$, when $p\neq 2$, this approximation turns out to be zero or a non-constant outer function.
Further, we determine the exact distance between the unit constant and any shift-invariant subspace and propose some open problems. Our results use the notion of Birkhoff-James orthogonality and Pythagorean Inequalities, along with an associated dual extremal problem, which leads to some interesting inequalities. Further consequences shed light on the lattice of shift-invariant subspaces of $H^p$, as well as the behavior of the zeros of optimal polynomial approximants in $H^p$.

\end{abstract}

\section{Introduction}\label{intro}
Optimal polynomial approximants (OPAs) are polynomials which heuristically approximate the reciprocal of an element in a function space. For example, let $X$ be a Banach space of analytic functions on a planar domain, for which polynomials are dense, that remains invariant under multiplication by polynomials. Also, let $\mathcal{P}_n$ be the set of all polynomials of degree at most $n$. Given $f \in X \setminus\{0\}$ and $n \in \N = \{0, 1, 2, \ldots\}$, we say that a polynomial $p_n$ in $\mathcal{P}_n$ is an $n$-th \textit{optimal polynomial approximant} to $1/f$ if it minimizes $\|1 - p f \|_X$ among all polynomials $p$ in $\mathcal{P}_n$. Note that OPAs always exist (since $f\cdot\mathcal{P}_n$ is a finite dimensional subspace of $X$), but they are not necessarily uniquely determined.


Under the name \textit{least squares inverses}, Robinson \cite{Ro} introduced these approximants, in a restricted context, as a way to address various problems in signal processing. Chui \cite{Chu1}, Chui and Chan \cite{ChuCha}, and Izumino \cite{Iz} followed the work of Robinson with related results. Fifty years later, starting with \cite{BCLSS1}, a renewed interest in the subject arose due to significant and interesting connections to function theory, reproducing kernels, and orthogonal polynomials (see, e.g., the surveys \cite{BC} and \cite{Sec}). OPAs were further investigated in Dirichlet-type spaces \cite{BFKSS,BKLSS}, in
more general reproducing kernel Hilbert spaces  
\cite{FMS}, in 
the spaces $\ell^p_A$ \cite{CRSX,ST}, in 
$L^p$ of the unit circle  \cite{Cent}, in 
Hilbert spaces of analytic 
functions on the bidisk  
\cite{BCLSS2,BKKLSS} and on the unit ball
\cite{SS2,SS1}, and in the context of free functions \cite{AAJS}. Recently, a long-standing conjecture for OPAs in several variables was disproved in \cite{BKS_Sh}. See also \cite{AS, BIMS, BMS, Felder, Felder2} for related work.

In this paper, we will focus on the Hardy spaces $H^p$, which for $ 0 < p < \infty$ can be defined as
\[
H^p:= \left\{ f \in \operatorname{Hol}(\D) :  \|f\|_p^p := \sup_{0 < r < 1} \frac{1}{2\pi} \int_0^{2 \pi} |f(r e^{i \theta})|^p  \, d \theta < \infty \right\},
\]
where $\operatorname{Hol}(\D)$ is the set of analytic functions on the open unit disc $\D$.
In the case $p = \infty$, we define
\[
H^{\infty}:= \left\{ f\in\operatorname{Hol}(\D) : \|f\|_{\infty} := \sup_{z \in \D} \left|f(z)\right| < \infty \right\}.
\]
For $ 1 \leq p \leq \infty,$ it is well known that $H^p$ is a Banach space and $H^2$ is a Hilbert space. Functions in $H^p$ have non-tangential limits almost everywhere on the unit circle $\T,$ and $H^p$ functions restricted to their boundary values on $\T$ can be viewed as a closed subspace of $L^p:=L^p(\T, dm)$, where $dm$ is normalized Lebesgue measure on $\T$. Indeed, $H^p$ can equivalently be defined as the set of $L^p$ functions with Fourier coefficients vanishing for negative frequencies.  For more on $H^p$ spaces see, e.g., \cite{Duren} or \cite{MR0133008}. 

This work was initially motivated by the question of whether OPAs in $H^p$ ($p \neq 2$) can have zeros in the closed unit disk, which, at the time of writing this manuscript, is still an open problem. 
This question arose from an interesting fact in $H^2$: if $f \in H^2$ is such that $f(0) \neq 0$, then the OPAs to $1/f$ can never vanish in the closed unit disk \cite{BKLSS_Lon,Chu1}. 

Note for $f \in H^p$, if $f(0) = 0$, then by subharmonicity, all OPAs in $H^p$ ($1 \leq p \leq \infty$) vanish identically. It is also worth noting that for $p = 1$ and $p = \infty,$ OPAs are not uniquely defined, unlike the case of $1<p<\infty$.  Therefore, for the rest of this work, we will consider only $ 1 < p < \infty$ and functions $f \in H^p$ such that $f(0) \neq 0$.

OPAs are intimately linked with the forward shift $S$,  given by $(Sf)(z) = zf(z)$;  the projection of the unit constant function onto the subspace $\operatorname{span}\{S^kf: k = 0, 1, \ldots, n \}$ is $f$ multiplied by the OPA of degree up to $n$. 
When $p=2$, this projection is an orthogonal projection and the Hilbert space structure makes computation of OPAs straightforward with linear algebra techniques.  However, when $p\neq 2$, the projection is a metric projection, which is nonlinear, and makes explicit computation of OPAs much more difficult. In this paper we find a strikingly simple expression for the \textit{limit} of OPAs in $H^p$ without calculating the OPAs explicitly. 
We will often use $z$ interchangeably with $S$, and say a subspace $M$ is $z$-invariant if $SM \subseteq M$. 
If $f\in H^p$ has corresponding OPAs $(p_n)_{n\ge0}$, then $\lim_{n \to \infty} p_nf$ can be seen as the projection of 1 onto the $z$-invariant subspace
\[
[f]_p := \overline{\operatorname{span}\{S^kf: k = 0, 1, 2, \ldots \}}^{H^p}.
\]
It is critical for us to recall the celebrated theorem of Beurling which describes every closed nontrivial $z$-invariant subspace $M$ of $H^2$ as $M = J \cdot H^2$, for some inner function $J$ (a function in $H^\infty$ is called inner if it has boundary values of unit modulus almost everywhere). Note that Beurling's Theorem holds for all $0 < p < \infty$ (see, e.g., \cite[p.\ 98]{MR628971}) and for weak-$*$ closed subspaces of $H^{\infty}$ (see, e.g.,  \cite[Theorem 7.5, Chapter 2]{MR628971}). A significant focus of this work is placed on studying projections of 1 onto these subspaces. 

As we shall see, these projections naturally lead to approximating (in $L^p$) conjugates of inner functions by $H^p$ functions.
Given an $H^p$ function, we determine the distance between $1$ and the $z$-invariant subspace of $H^p$ generated by that function. In $H^2$, that distance is measured via the orthogonal projection of $1$ onto the invariant subspace.  Indeed, that orthogonal projection gives a scalar multiple of the generating inner function.  This approach gives rise to a proof of Beurling's Theorem for $H^2$ (see, e.g., \cite[p.~100]{MR0133008}). Again, we note for $p \neq 2$, $H^p$ is not a Hilbert space but rather a Banach space and in this setting we no longer have orthogonal projections, but rather \textit{metric} projections. Although these projections are non-linear,  we can still discuss orthogonality in the Birkhoff-James sense. Our immediate aim is to explore several questions, including: 
\begin{itemize}
\item[-] What is the metric projection of $1$ onto a $z$-invariant subspace in $H^p$?
\item[-] Is the projection of $1$ onto a $z$-invariant subspace an inner function?
\item[-] If the projection of $1$ onto a $z$-invariant subspace is not an inner function, does the projection have an inner factor which is not shared with the generating function for the subspace? 
\item[-] Given a $z$-invariant subspace, what is the distance between $1$ and that invariant subspace? 
\end{itemize}    

When considering metric projections, the use of duality is a critical tool for the related extremal problem (see, e.g.,  \cite{Duren,SYaK}).  In our context, the statement of duality is as follows.  Let $1 < p < \infty,$  $\psi \in L^p$, and $q = \frac{p}{p-1}$.  Then the distance between $\psi$ and $H^p$ can be expressed as the norm of the linear functional given by $\psi$ on the annihilator of $H^p$ inside $L^q,$  namely, 
\begin{equation}\label{duality}
\sup_{\|f\|_{q} \leq 1} \left| \frac{1}{2 \pi i} \int_{\T} f(\zeta) \psi(\zeta)   \, d \zeta \right| = \inf_{\varphi \in H^p} \| \psi - \varphi \|_{L^p}.
\end{equation}

The paper is organized as follows:
\begin{itemize}
\item[-]In Section \ref{BJ}, we discuss the background needed to pass from the Hilbert setting to the more general Banach setting. In particular, we introduce Birkhoff-James orthogonality and state some useful inequalities, known as Pythagorean inequalities, which will be used in the subsequent sections.
\item[-]In Section \ref{MetP}, we answer the questions posed above by finding an explicit formula for the metric projection of $1$ onto any $z$-invariant subspace of $H^p$ and calculating the associated distance. This leads to several corollaries that enrich our understanding of the lattice of $z$-invariant subspaces of $H^p$.
\item[-]In Section \ref{opas}, we study the zeros of OPAs in $H^p$ and outline a possible path to prove that those zeros always lie outside $\D$ for any $p \neq 2$.  We also consider planar disks, centered at the origin, in which OPAs cannot vanish. We establish several estimates for the radius of such a disk, depending only on $p$ and $f$, for which OPAs are zero-free. 
\item[-]Finally, in Section \ref{concl}, we conclude with some comments and open questions. 
\end{itemize}

\section{Preliminaries: Birkhoff-James Orthogonality and Pythagorean Inequalities\color{black}}\label{BJ}

Recall that in $H^2$, if $M$ is a closed nontrivial $z$-invariant subspace of $H^2$, then the orthogonal projection of $1$ onto $M$ gives a scalar multiple of the inner function that is the generator for $M$. We seek to understand the properties of the analagous metric projection when $1 < p < \infty$. 
Recall that the metric projection of $1$ onto a $z$-invariant subspace $M \subseteq H^p$ is the unique function $g_p^*$ minimizing $\| 1 - g\|_p$ over all $g \in M$. We are particularly interested in obtaining an explicit formula for this projection and determining whether or not it is a generator of $M$. Since by Beurling's Theorem we know that $M = [J]_p$ for some inner function $J$, we are equivalently asking if $g_p^* = J \cdot G$ where $G$ is an outer function. To examine this question, it will be useful to have some background on metric projections.

Let $\mathbf{x}$ and $\mathbf{y}$ be vectors belonging to a normed linear space $\mathscr{X}$.  We say that $\mathbf{x}$ is \textit{orthogonal} to $\mathbf{y}$ in the Birkhoff-James sense   if
 \[    \|  \mathbf{x} + \beta \mathbf{y} \|_{\mathscr{X}} \geq \|\mathbf{x}\|_{\mathscr{X}}  \]
for all scalars $\beta$ \cite{Bir,Jam}. 
In this situation we write $\mathbf{x} \perp_{\mathscr{X}} \mathbf{y}$.
This way of generalizing orthogonality is particularly useful in our context since it is based on an extremal condition.

If $\mathscr{X}$ is an inner product space, then the relation $ \perp_{\mathscr{X}} $ coincides with the usual orthogonality relation.  In more general spaces, however, the relation $\perp_{\mathscr{X}}$ is neither symmetric nor linear.  In the case $\mathscr{X} = L^p(\mu)$, for a measure $\mu$,  let us write $\perp_p$ instead of  $\perp_{L^p}$.
When $1<p<\infty$, Birkhoff-James orthogonality in $L^p$ can be expressed as in integral condition, which will be useful later. 

\begin{Theorem}[James (1947)\ \cite{Jam}]\label{praccritperp}
Suppose that $1<p<\infty$.  If $f$ and $g$ are elements of $L^p$, then 
\begin{equation}\label{BJp}
 {f} \perp_{p} {g}  \iff   \int |f|^{p - 2} \overline{f} g \,d\mu  = 0,
\end{equation}
where any occurrence of ``$|0|^{p - 2} \overline{0}$'' in the integrand is interpreted as zero.
\end{Theorem}

For a short proof of this criterion for orthogonality in $L^p$, see \cite[Theorem A, p.\ 124]{DuSchu}. For an extension to more general normed spaces, see \cite{Bea}.
In light of this integral orthogonality condition, we define, for a measurable function $f$ and any $s > 0$, the notation
\begin{equation}\label{s-p}
f^{\langle s \rangle} := |f|^{s-1}\overline{f}.
\end{equation}
Note that if $f \in L^p$ for $1 < p < \infty,$ then $f^{\langle p-1 \rangle} \in L^q,$ where $p$ and $q$ are conjugate exponents, and then for $g \in L^p,$
\[
f \perp_{p} g \iff \langle g, f^{\langle p - 1\rangle} \rangle = 0,
\]
where $\langle \cdot \ , \cdot \rangle$ is the standard dual pairing between $L^p$ and its dual.
Consequently, the relation $\perp_{p}$ is linear in its second argument, and it then makes sense to speak of a vector being orthogonal to a subspace of $L^p$.  In particular, if $f \perp_p g$ for all $g$ belonging to a subspace $\mathscr{M}$ of $L^p$, then
\[
       \| f + g \|_{L^p}  \geq  \|f\|_{L^p}
\]
for all $g \in \mathscr{M},$ and thus  $\operatorname{dist}_{L^p}(f, \mathscr{M}) = \|f\|_{L^p}$.  In other words, the best approximation to $f$ in $\mathscr{M}$ is $0$, or, using another terminology, $f$ is badly approximable by $\mathscr{M}$. In our case, we will be interested in $f = 1 - g^*,$ where $g^*$ is the metric projection of $1$ onto $\mathscr{M}$, and therefore, to identify $g^*$, we are looking for $g^* \in \mathscr{M}$
such that $(1-g^*) \perp_p g$ for every $g \in \mathscr{M}$.

There is a version of the Pythagorean Theorem for $L^p$, where orthogonality is in the Birkhoff-James sense.  It takes the form of a family of inequalities relating the lengths of orthogonal vectors with that of their sum.

\begin{Theorem}\label{pythagthm}
Suppose $f \perp_p g$ in $L^p$.
If $p \in (1, 2]$, then
\begin{align}
   \| f + g \|^p_{L^p} & \leq  \|f\|^p_{L^p} + \frac{1}{2^{p-1}-1}\|g\|^p_{L^p} \label{upper1}\\
    \| f + g \|^2_{L^p} & \geq  \|f\|^2_{L^p} + (p-1)\|g\|^2_{L^p}.\label{lower1}
\end{align}
If $p \in [2, \infty)$, then
\begin{align}
   \| f + g \|^p_{L^p} & \geq  \|f\|^p_{L^p} + \frac{1}{2^{p-1}-1}\|g\|^p_{L^p} \label{lower2}\\
    \| f + g \|^2_{L^p} & \leq  \|f\|^2_{L^p} + (p-1)\|g\|^2_{L^p}.\label{upper2}
\end{align}
\end{Theorem}

These inequalities originate from \cite{Byn,BD,CMP1}; see \cite[Corollary 3.4]{CR} for a unified treatment with broader classes of spaces.

 It will be expedient to refer to \eqref{upper1} and \eqref{upper2} as the upper Pythagorean inequalities, and \eqref{lower1} and \eqref{lower2} as the lower Pythagorean inequalities, so that the two cases depending on $p$ can be handled together.  The specific values of the positive multiplicative constants, i.e., $p-1$ and $1/(2^{p-1}-1)$, are generally unimportant, and thus they will usually be denoted simply by $K$.

\section{Metric projections of $1$ onto invariant subspaces}\label{MetP}

\noindent For $f \in H^p$, recall that $[f]_p$ denotes the closure of the polynomial multiples of $f$ in $H^p$. Colloquially, we say $[f]_p$ is the closed $z$-invariant subspace generated by $f$.  
The present section deals with solving the minimization problem
\[
\inf_{g \in M} \| 1 - g\|_p,
\]
where $M$ is a $z$-invariant subspace. 
We will use several standard facts concerning $H^p$ functions; e.g., if $J\in H^p$ is inner, then by Beurling's Theorem, $[J]_p = J \cdot H^p$, a function $f \in H^p$ is outer if and only if $[f]_p = H^p$, and a polynomial is outer if and only if it does not vanish in the open unit disk. See \cite{Duren} for more properties of $H^p$ functions. 

\subsection{Extremal functions and distances}
We first consider the case where the $z$-invariant  subspace $M$ is generated by a finite Blaschke product. 

\begin{Proposition}\label{Npoint}
       Let $1<p<\infty$. If  $J$ is the finite Blaschke product 
       \[
             J(z) = \prod_{k=1}^{N}\frac{ z - a_k }{ 1 - \bar{a}_k z }
       \]
       with zeros $a_1, a_2,\ldots, a_N\in \mathbb{D}\setminus\{0\}$,   then
       \[
            \inf_{h\in H^p}\|1 - Jh\|_p
       \] 
       is attained when $h = h^*$, where  for $ N> 1$,
    \begin{equation}\label{jstarform}
          1 - J(z)h^*(z)  =  (1 - |a_1 a_2 \cdots a_N|^2)^{2/p}  \Bigg[ \frac{\prod_{k=1}^{d}(1 - \bar{w}_k z)}{\prod_{k=1}^{N}(1 - \bar{a}_k z)}\Bigg]^{2/p},\ z \in \mathbb{D}
       \end{equation}
       and the parameters $w_1, w_2,\ldots w_d \in \overline{\mathbb{D}}\setminus \{0\}$ satisfy
       \begin{equation}\label{consconds}
            1 = (1 - |a_1 a_2 \cdots a_N|^2)\frac{\prod_{k=1}^{d}(1 - \bar{w}_k a_j)}{\prod_{k=1}^{N}(1 - \bar{a}_k a_j)}
       \end{equation}
       for all $j$, $1\leq j \leq N$, where $1 \leq d \leq N-1$. 
       For $N = 1$, \begin{equation}\label{jstarform1}
          1 - J(z)h^*(z)  =  \Bigg[ \frac{1 - |a_1|^2}{1 - \bar{a}_1 z}\Bigg]^{2/p},\ z \in \mathbb{D}.
       \end{equation}         
       Moreover, $\operatorname{dist}_{H^p} (1, [J]_p) = \left(1 - |J(0)|^2 \right)^{1/p}$.
       
\end{Proposition}

\begin{proof}
The function
\[
      \overline{ a_1 a_2 \cdots a_N}J(z) = \overline{ a_1 a_2 \cdots a_N} \prod_{k=1}^{N}\frac{z - a_k}{1-\bar{a}_k z}
\]
takes values in a disk of radius $|a_1 a_2 \cdots a_N|<1$.  Consequently the function
\begin{equation*}
      1 -  (-1)^N\overline{ a_1 a_2 \cdots a_N} \prod_{k=1}^{N}\frac{z - a_k}{1-\bar{a}_k z}  = 1 - \overline{J(0)}J(z)
\end{equation*}
is outer.  It remains outer if we multiply through by $\prod_{k=1}^{N}(1 - \bar{a}_k z)$.  That is to say,
\begin{equation}\label{roastchicken}
     \prod_{k=1}^{N}(1 - \bar{a}_k z) - (-1)^N \overline{ a_1 a_2 \cdots a_N} \prod_{k=1}^{N}(z - a_k)
\end{equation}
is an outer polynomial.  In fact, the degree $N$ terms cancel, leaving an outer polynomial of degree $d\leq N-1$. If $N = 1$, we simply obtain the constant $1 - |a_1|^2$ and the proof is complete for that case. Otherwise,  we write  the outer polynomial \eqref{roastchicken} as
\begin{equation}\label{outerpoly}
      c(1 - \bar{w}_1 z)(1 - \bar{w}_2 z)\cdots(1 - \bar{w}_{d}z).
\end{equation}
Because this polynomial is outer, each $w_k$ satisfies $0<|w_k| \leq 1$.   Substituting $z = a_j$ into the equation
\[
     \prod_{k=1}^{N}(1 - \bar{a}_k z) - (-1)^N \overline{ a_1 a_2 \cdots a_N} \prod_{k=1}^{N}(z - a_k) = c(1 - \bar{w}_1 z)(1 - \bar{w}_2 z)\cdots(1 - \bar{w}_{d}z)
\]
yields
\begin{align}
      \prod_{k=1}^{N}(1 - \bar{a}_k a_j) - (-1)^N \overline{ a_1 a_2 \cdots a_N} \prod_{k=1}^{N}(a_j - a_k) &= c(1 - \bar{w}_1 a_j)(1 - \bar{w}_2 a_j)\cdots(1 - \bar{w}_{d}a_j) \nonumber \\
       \prod_{k=1}^{N}(1 - \bar{a}_k a_j) -0  &= c \prod_{k=1}^{d}(1 - \bar{w}_k a_j) \nonumber \\
       1 &= c\frac{\prod_{k=1}^{d}(1 - \bar{w}_k a_j)}{\prod_{k=1}^{N}(1 - \bar{a}_k a_j)}  \label{masalachai}
\end{align} for each $j$, $1\leq j\leq N$.

The rational expression 
\[
    \frac{\prod_{k=1}^{d}(1 - \bar{w}_k z)}{\prod_{k=1}^{N}(1 - \bar{a}_k z)}
\]
is outer, so that its $2/p$ power is an analytic function in $\mathbb{D}$.  
The condition \eqref{masalachai} tells us that the $H^p$ function
\begin{equation*}
    1 - \Bigg[c\frac{\prod_{k=1}^{d}(1 - \bar{w}_k z)}{\prod_{k=1}^{N}(1 - \bar{a}_k z)}\Bigg]^{2/p}  = 1 - \left( 1 - \overline{J(0)} J(z) \right)^{2/p}
\end{equation*}
has zeros at $a_1, a_2,\ldots, a_N$, and thus can be written as $Jh$ for some $h \in H^p$.
  Rearranging gives
\begin{equation}\label{Jhat3}
     1 - J(z)h(z) = \Bigg[c\frac{\prod_{k=1}^{d}(1 - \bar{w}_k z)}{\prod_{k=1}^{N}(1 - \bar{a}_k z)}\Bigg]^{2/p} = \left( 1 - \overline{J(0)} J(z) \right)^{2/p},
\end{equation}
an outer function in $H^p$.

Recalling notation from \eqref{s-p}, we have, for $|z|=1$,
 \begin{align}
     (1- J h)^{\langle p-1 \rangle}J &=  \left\{ \Bigg[ c\frac{\prod_{k=1}^{d}(1 - \bar{w}_k z)}{\prod_{k=1}^{N}(1 - \bar{a}_k z)}\Bigg]^{2/p}\right\}^{\langle p-1 \rangle} \prod_{k=1}^{N}\frac{(z - a_k) }{ (1-\bar{a}_k z) } \nonumber \\
    &=  \left\{ \Bigg[ c\frac{\prod_{k=1}^{d}(1 - \bar{w}_k z)}{\prod_{k=1}^{N}(1 - \bar{a}_k z)}\Bigg]^{2/p}\right\}^{(p-2)/2} 
     \left\{\overline{ \Bigg[ c\frac{\prod_{k=1}^{d}(1 - \bar{w}_k z)}{\prod_{k=1}^{N}(1 - \bar{a}_k z)} } \Bigg]^{2/p}\right\}^{1+(p-2)/2}
     \prod_{k=1}^{N}\frac{(z - a_k) }{ (1-\bar{a}_k z) } \nonumber \\
     &=   \Bigg[ c\frac{\prod_{k=1}^{d}(1 - \bar{w}_k z)}{\prod_{k=1}^{N}(1 - \bar{a}_k z)}\Bigg]^{(p-2)/p}
      \bar{c}\frac{\prod_{k=1}^{d}(1 - {w}_k \bar{z})}{\prod_{k=1}^{N}(1 - {a}_k \bar{z})}
     \cdot\prod_{k=1}^{N}\frac{(z - a_k) }{ (1-\bar{a}_k z) }   \nonumber\\
     &=   \Bigg[ c\frac{\prod_{k=1}^{d}(1 - \bar{w}_k z)}{\prod_{k=1}^{N}(1 - \bar{a}_k z)}\Bigg]^{(p-2)/p} 
      \bar{c}z^{N-d} \frac{\prod_{k=1}^{d}(z - {w}_k )}{\prod_{k=1}^{N}(z - {a}_k)}
     \cdot\frac{ \prod_{k=1}^{N}(z - a_k) }{ \prod_{k=1}^{N}(1-\bar{a}_k z) } \nonumber \\
     &=  z^{N-d} \Bigg[ c\frac{\prod_{k=1}^{d}(1 - \bar{w}_k z)}{\prod_{k=1}^{N}(1 - \bar{a}_k z)}\Bigg]^{(p-2)/p}
      \bar{c} \frac{\prod_{k=1}^{d}(z - {w}_k )}{ \prod_{k=1}^{N}(1-\bar{a}_k z)}, \label{spicyham}
 \end{align}
 which is $z^{N-d}$, multiplied by an element of $H^q$.  Notice also that the exponent $N-d$ is an integer greater than or equal to $1$.  These conditions imply that, for all $n\geq 0$,
 \[
      \int_{0}^{2\pi} (1- J(e^{i\theta}) h(e^{i\theta}))^{\langle p-1 \rangle}J(e^{i\theta}) e^{in\theta}\,\frac{d\theta}{2\pi} = 0,\ 
 \]
 and hence $1-Jh \perp_p J z^n$ for all indices $n\geq 0$.  We may conclude that $h= h^*$ is indeed the extremal function for which we are looking. 
 
By comparing \eqref{roastchicken} and \eqref{outerpoly} when $z=0$, we obtain $c = 1 - |a_1 a_2 \cdots a_N|^2$.  
Moreover, from \eqref{Jhat3}, 
\begin{align*}
\operatorname{dist}_{H^p} (1, [J]_p) &= \| 1 - J h^*\|_p\\
&= \left\|\left( 1 - \overline{J(0)} J \right)^{2/p}\right\|_p \\
&=  \|  1 - \overline{J(0)} J  \|_2^{2/p} \\
&= \left(1 - |J(0)|^2 \right)^{1/p},
\end{align*}
as claimed. 
\end{proof}

We note that Proposition \ref{Npoint} provides explicit information about the constants $w_k$. 
The solution (\ref{jstarform}) to the extremal problem is consistent with equation (11) on page 138 of Duren \cite{Duren}, which solves the general dual extremal problem in $H^p$ with rational kernels.  The benefit of the present approach, in which the kernel arises in connection with a finite Blaschke product, is that condition (\ref{consconds}) enables the direct calculation of the parameters $w_k$, $1 \leq k \leq N-1$, as well as the scaling factor; on the other hand, in \cite{Duren} the determination of these parameters is left open as a ``very difficult problem.''
For $N = 1,$ the constants $w_k$ are absent and the formula \eqref{jstarform1} is a normalized power of the Szeg\"{o} kernel.

Using the following lemma, we can extend the hypotheses of Proposition \ref{Npoint} to include any inner function.

\begin{Lemma}\label{FBPapprox}
	Let $1<p<\infty$ and $J$ be an inner function. Then there exists a sequence $(J_n)_{n\ge0}$ of finite Blaschke products such that $\|J_n-J\|_p\to0$ as $n \to \infty$.
\end{Lemma}
\begin{proof}
	Firstly, it is known that the set of (infinite) Blaschke products is dense in the set of inner functions with respect to the $H^\infty$ norm, and thus also the $H^p$ norm (this is a corollary of a theorem of Frostman; see \cite[Corollary 6.5]{MR628971}). Thus, it suffices to consider the case where $J$ is an infinite Blaschke product
	\[
	J(z) = \zeta\prod_{k=1}^{\infty}  \frac{|a_k|}{a_k}\frac{ a_k -z }{ 1- \bar{a_k} z }, \ \ z\in\D,
	\] 
	where $\zeta\in\mathbb{T}$ and $(a_k)_{k=1}^\infty$ is a Blaschke sequence in $\D$. For $n\in\mathbb{N}$ we set
	\[
	J_n(z)  = \zeta\prod_{k=1}^{n}\frac{|a_k|}{a_k}\frac{ a_k -z }{ 1- \bar{a_k} z }, \ \ \ z\in\D.
	\]
	We know that $\|J_n-J\|_2\to0$ (see the lemma on p.\ 64 in \cite{MR0133008} and the discussion thereafter). Thus, the result follows immediately if $1 < p < 2$. If $p>2$, then, almost everywhere on the circle, we have 
    \[
    |J_n-J|^p=|J_n-J|^2\cdot|J_n-J|^{p-2}\leq|J_n-J|^2\cdot2^{p-2}
    \]
    and so $\|J_n-J\|_p\leq\|J_n-J\|_2^{\frac{1}{p}}\cdot2^{\frac{p-2}{p}}$, which gives the desired result.
\end{proof}

We now extend Proposition \ref{Npoint}.

\begin{Theorem}\label{fstarnozero_0}
Let $1<p<\infty$.  Suppose $J$ is an inner function with $J(0)\neq 0$, and 
	\begin{equation}\label{ext_prob}
	\inf_{h\in H^p}\|1 - Jh\|_p
	\end{equation}
	is attained when $h = h^*$.  Then $[Jh^*]_p = [J]_p$ and, in particular, $h^*$ is outer.  Moreover, 
    \[
    \operatorname{dist}_{H^p} (1, [J]_p) = \left(1 - |J(0)|^2 \right)^{1/p}.
    \]

\end{Theorem}

\begin{proof}
By Lemma \ref{FBPapprox}, we can find a sequence $(J_n)_{n\ge0}$ of finite Blaschke products converging to $J$ in $H^p$.
Hence, by continuity of metric projections, we have $\operatorname{dist}_{H^p} (1, [J_n]_p) \to \operatorname{dist}_{H^p} (1, [J]_p)$ as $n\to\infty$ (see \cite[Proposition 4.8.1]{CMR}).
By Proposition \ref{Npoint}, 
\begin{equation*}
\operatorname{dist}_{H^p} (1, [J_n]_p) = \left(1 - |J_n(0)|^2 \right)^{1/p}, 
\end{equation*}
but since $J_n \rightarrow J$ in $H^p$, certainly $J_n(0) \rightarrow J(0),$ and therefore 
\begin{equation}\label{distance}
\operatorname{dist}_{H^p} (1, [J]_p) = \left(1 - |J(0)|^2 \right)^{1/p}.
\end{equation}
Now let us show that $[Jh^*]_p = [J]_p$, i.e., that $h^*$ is outer. Suppose $h^* = J_1 g$ for some nontrivial inner function $J_1$ and $g \in H^p$ outer. Then by the above established distance formula (\ref{distance}), and using the fact that $|J_1(0)| < 1$, we obtain 
\begin{align*}
\| 1 - Jh^* \|_p^p &= 
\inf_{h \in H^p} \| 1 - Jh \|_p^p\\
&\geq \inf_{g \in H^p} \| 1 - J J_1 g \|_p^p\\
&=  1 - |J(0)|^2 |J_1(0)|^2\\
&> 1 - |J(0)|^2\\
&= \| 1 - Jh^* \|_p^p,
\end{align*}
which is a contradiction, and therefore $h^*$ must be outer.
\end{proof}

Note for $1 < p < \infty$, $p \neq 2$, the distance formula in Theorem \ref{fstarnozero_0} appears to be previously unknown. This formula reveals additional information concerning the lattice of invariant subspaces of $H^p$ in that it allows for quantitative comparison among $z$-invariant subspaces based on the generator of the subspace.
In addition to this observation, the proof of Proposition \ref{Npoint} gives rise to a guess for the metric projection of $1$ onto any proper invariant subspace. This observation, fortuitously, allows us to streamline the proof of a more general theorem that encompasses the previously established results of this section.

\begin{Theorem}\label{s_and_s}
Let $1 < p < \infty$ and $f \in H^p$, $f(0) \neq 0$. Put $f = J F$, with $J$ inner and $F$ outer, and let $\hat{J} = \overline{J(0)}J$. Let  $g_p^*$ be the metric projection of $1$ onto $[f]_p$, that is, $g_p^*$ is the unique solution to the minimization problem 
\[
\inf_{g \in [f]_p} \| 1 - g \|_p.
\]
Then $g_p^*$ is given as
\[
g_p^* = 1 - (1 - \hat{J})^{2/p}.
\]
Moreover, $g_p^*$ has no inner factor other than $J$ and
\[
\operatorname{dist}_{H^p}(1, [f]_p) = (1 - |J(0)|^2)^{1/p}. 
\]
\end{Theorem}

\begin{proof}
Note that for $z \in \D,$  $|\hat{J}(z)| \leq |J(0)| < 1 $ and therefore $1 - \hat{J}$ is non-vanishing in the disk, and so $g_p^*$ is a well-defined function in $H^p$.  We would like to show that $g_p^*$ is the metric projection of $1$ onto $[f]_p$. 
By Beurling's Theorem for $H^p$, it suffices to consider the infimum of $\|1 - g\|_p$ for $g \in [J]_p$ and show that the Birkhoff-James orthogonality conditions hold:
\[
1 - g_p^* \perp_p z^k J \ \  \forall k \ge 0.
\]
Observe, for $k \ge 0$, 
\begin{align*}
\int_\T\left| 1 - g_p^* \right|^{p-2} (1 - g_p^*)\ \overline{z^kJ} \,dm
&= \int_\T \left| (1 - \hat{J})^{2/p} \right|^{p-2} \ (1 - \hat{J})^{2/p} \ \overline{z^kJ} \,dm\\
&= \int_\T (1 - \hat{J}) \  \overline{(1 - \hat{J})}^{(p-2)/p} \ \overline{z^kJ} \,dm\\
&= \int_\T \overline{(1 - \hat{J})}^{(p-2)/p} \ \overline{z^kJ}\,dm - \overline{J(0)}\int_\T \overline{(1 - \hat{J})}^{(p-2)/p} \ \overline{z^k} \,dm\\
&= 0.
\end{align*}
Further, 
\[
\| 1 - g_p^* \|_p^p =\| (1 - \hat{J})^{2/p}\|_p^p = \| 1-  \hat{J}\|_2^2 = 1 - |J(0)|^2,
\]
 that is, 
\[
\operatorname{dist}_{H^p}(1, [f]_p) = (1 - |J(0)|^2)^{1/p}. 
\]
Now, by the same argument as in the end of Theorem \ref{fstarnozero_0}, $g_p^*$ has no additional inner factor besides $J$.
\end{proof}

\subsection{Consequences and corollaries}

The distance formula appearing in Theorem \ref{s_and_s} tells us that an additional inner factor \textit{strictly} increases the distance between 1 and the corresponding invariant subspace. Namely, we have the following:
\begin{Corollary}\label{nested_spaces}
Let $ 1 < p < \infty$ and let $M$ be a nontrivial closed $z$-invariant subspace of $H^p$.  Let $N \subsetneq M$ be a strictly smaller invariant subspace.  Then 
\begin{equation*}
\operatorname{dist}_{H^p}(1, N)
     > \operatorname{dist}_{H^p}(1, M).
\end{equation*}
\end{Corollary}


Let us employ this corollary in an example. 
\begin{Example}
For each $n\ge2$, let $B_n$ be the Blaschke product having precisely a zero of multiplicity $n$ at $1 - 1/n$, and zeros nowhere else. One may check that $B_n$ converges pointwise to the atomic singular inner function $e^{-\frac{1+z}{1-z}}$.  Moreover, it is easy to see that $|B_n(0)| < |B_{n+1}(0)|$.
Therefore, the invariant subspace generated by the singular inner function is closer to $1$ than any of the subspaces generated by the associated Blaschke products which approximate the singular inner function. This is an example of a more general phenomenon, namely the distance of $1$ to an invariant subspace generated by a singular inner function is often smaller, in the long run, than the distance between $1$ and the invariant subspaces generated by Blaschke products converging to that singular inner function.  This sheds additional light on the structure of the lattice of $z$-invariant subspaces of $H^p$.
\end{Example}

We note also that if $f = J F \in H^p$ where $J$ is a nontrivial inner function and $F$ is outer, then 
\begin{equation*}
\operatorname{dist}_{H^p}(1, [f]_p) = \inf_{g \in H^p} \| 1 - Jg \|_p^p = \inf_{g \in H^p} \| \overline{J} - g \|_{L^p}^p.
\end{equation*}
Therefore, the problem of finding the metric projection of $1$ in $H^p$ onto the invariant subspace generated by an inner function is the same as the problem of best approximation of the conjugate of that inner function in $L^p$ by an $H^p$ function. Theorem \ref{s_and_s} tells us that the best approximation is an outer function. Phrasing the problem in this way also naturally brings in the tool of duality in extremal problems.
Applying previously mentioned duality \eqref{duality} to this particular extremal problem, Theorem \ref{s_and_s} gives rise to a highly nontrivial inequality for the dual problem, which we record now. 

\begin{Corollary}\label{J_Ineq}
    Let $ 1 < q < \infty$. If $J_1$ and $J_2$ are non-constant inner functions not vanishing at the origin, then the following strict inequality holds:
    \begin{equation*}
        \sup_{\|g\|_q \leq 1} \left| \frac{1}{2 \pi i} \int_{\T} g \overline{J_1} \, d \zeta \right| <    \sup_{\|g\|_q \leq 1} \left| \frac{1}{2 \pi i} \int_{\T} g \overline{J_1 J_2} \, d \zeta \right|.
    \end{equation*}
\end{Corollary}

\begin{proof}
 The result follows from Corollary \ref{nested_spaces} and duality \eqref{duality}. 
\end{proof}
Note that Corollary \ref{J_Ineq} does not hold for $q = 1$, since by Poreda's Theorem \cite{Po76} for any finite Blaschke product $B$, 
\begin{equation*}
    \inf_{f \in H^{\infty}} \|\overline{B} - f \|_{L^{\infty}} = 1.
\end{equation*}
Then by duality, if $J_1$ and $J_2$ are finite Blaschke products, both supremum values in Corollary \ref{J_Ineq} are equal to one.

Applying Corollary \ref{J_Ineq} to Blaschke products, we obtain the following non-trivial inequality:

\begin{Corollary}\label{dual_inequality}
Let $ 1 < q < \infty$. Let $\{ a_k\}_{k=1}^{\infty}$ be a sequence of distinct values in $\D\setminus \{0\}$ and let $n > m$. Then the following strict inequality holds:
\begin{equation*}
 \sup_{\|g\|_q =1}  \left| \sum_{j=1}^m g(a_j) \left( \frac{\prod_{k=1}^m (1 - \overline{a_k} a_j )}{\prod_{k=1, k \neq j}^m (a_k - a_j)} \right) \right| < \sup_{\|g\|_q =1}  \left| \sum_{j=1}^n g(a_j) \left( \frac{\prod_{k=1}^n (1 - \overline{a_k} a_j )}{\prod_{k=1, k \neq j}^n (a_k - a_j)} \right) \right|. 
\end{equation*}
\end{Corollary}
\begin{proof}
Let $J_1(z) = \prod_{k=1}^m \frac{a_k -z}{1 - \overline{a_k}z}$.  Then for any $g \in H^q,$ since $\overline{J_1(\zeta)} = \frac{1}{J_1(\zeta)}$ for $|\zeta| = 1$ and by the Residue Theorem,
\begin{equation*}
    \frac{1}{2 \pi i} \int_{\T} g \overline{J_1} \, d \zeta = \frac{1}{2 \pi i} \int_{\T} g(\zeta) \left( \prod_{k=1}^m \frac{1 - \overline{a_k}\zeta}{a_k -\zeta} \right) \, d \zeta = \sum_{j=1}^m \operatorname{Res }(\tilde{g}; a_j),
\end{equation*}
where $\tilde{g}(z) = g(z) \left( \prod_{k=1}^m \frac{1 - \overline{a_k}z}{a_k -z} \right)$.
Note that 
\begin{equation*}
    \operatorname{Res }(g; a_j) = g(a_j) \left( \frac{\prod_{k=1}^m (1 - \overline{a_k} a_j )}{\prod_{k=1, k \neq j}^m (a_k - a_j)} \right).
\end{equation*}
Doing a similar computation for the right hand side of the inequality yields the result.
\end{proof}

Clearly, one can easily extend Corollary \ref{dual_inequality} to infinite Blaschke products $B_1$ and $B_2$ having distinct zeros and such that the zero set of $B_1$ is strictly contained in the zero set of $B_2$, and the zero set of $B_2$ does not contain the origin.
When we apply Corollary \ref{dual_inequality} to two Blaschke factors, we obtain the following:

\begin{Example}
Suppose $a, b \in \D \setminus \{0\}$ are distinct and let $J_1(z) = \frac{z-a}{1-\bar{a}z}$ and $J_2(z)=\frac{z-b}{1-\bar{b}z}$. Then the left hand side of the inequality in Corollary \ref{dual_inequality} is
\begin{equation*}
 \sup_{\|g\|_q =1} |g(a)| (1 - |a|^2).
\end{equation*}
It is easy to show using H\"{o}lder's inequality that this supremum is equal to 
$(1 - |a|^2)^{\frac{1}{p}}$ (here, $p$ is the H\"{o}lder conjugate to $q$) and the extremal function is 
\[
g^*(z) = \left( \frac{1 - |a|^2}{(1 - \bar{a}z)^2}\right)^{1/q}.
\]
On the other hand, the right hand side of the inequality in Corollary \ref{dual_inequality} is
\begin{equation*}
 \sup_{\|g\|_q =1}
\left| g(a)(1-|a|^2) \frac{1 - \bar{b}a}{b-a} - g(b)(1-|b|^2)\frac{1 - \bar{a}b}{b-a} \right|.
\end{equation*}
Therefore, Corollary \ref{dual_inequality} states in this special case that the following strict inequality holds:
\begin{equation*}
(1 - |a|^2)^{1/p} < \sup_{\|g\|_q =1}
\left| g(a)(1-|a|^2) \frac{1 - \bar{b}a}{b-a} - g(b)(1-|b|^2)\frac{1 - \bar{a}b}{b-a} \right|.
\end{equation*}
Curiously, even for $p = q = 2$, we have not been able to find a direct proof of this elementary inequality for distinct arbitrary elements $a, b\in \D\setminus\{0\}$.
\end{Example}

The metric projections discussed in the present section can also be obtained via limits of optimal polynomial approximants, which we discuss now.

\section{Optimal polynomial approximants in $H^p$}\label{opas}
For $n\in\N$, we denote by $\mathcal{P}_n$ the set of complex polynomials of degree at most $n$. Given $1<p<\infty$ and $f\in H^p$ with $f(0) \neq 0$, there exists a unique polynomial $q_{n,p}[f]\in\mathcal{P}_n$ such that
\[
\|1-q_{n,p}[f]f\|_p = \inf_{q\in\mathcal{P}_n}\|1-qf\|_p.
\]
We recall that these minimizing polynomials are called the \textit{optimal polynomial approximants} (OPAs) to $1/f$ in $H^p$. These polynomials have been extensively studied in various settings (for existence and uniqueness of OPAs in Hardy spaces see \cite[Section~2]{Cent}; see Section \ref{intro} for references pertaining to other relevant work). 

It is well known that if $f\in H^2$ with $f(0)\neq0$, then the OPAs to $1/f$ in $H^2$ cannot vanish in the closed unit disk. This result was first established in \cite{Chu1} and later reestablished in \cite{BKLSS_Lon}. On the other hand, a corresponding result for $p\neq2$ has been explored (see \cite{Cent, CCF, ChengFelder}) but has yet to be fully understood.  We begin with an alternative proof of this fact for $H^2$ which may be helpful in extending the result to the $H^p$ setting.

\begin{Proposition}\label{DI_p2}
Let $f \in H^2$ with $f(0) \neq 0$. If $J$ is any non-constant inner function, then there exists a constant $c = c_{f}\in\mathbb{T}$ such that 
\begin{equation*}
    \| 1 - J f \|_2 > \|1 - cf\|_2.
\end{equation*}

\end{Proposition}
\begin{proof}
Observe that 
    \begin{align*}
    \| 1 - J f \|_2^2
    &= \int_\T\left| 1 - Jf \right|^{2}\, dm\\
    &= \int_\T (1 - Jf)\overline{(1 - Jf)}\, dm\\
    &= 1+\|f\|^2-2\operatorname{Re}(J(0)f(0))\\
    &\geq 1+\|f\|^2-2|J(0)f(0)|\\
    &>1+\|f\|^2-2|f(0)|,
\end{align*}
where the penultimate inequality holds because, by the Maximum Principle, $|J(0)|<1$.
Taking $c$ so that $cf(0)=|f(0)|$, a simple calculation, similar to that above, shows 
\begin{equation*}
    \| 1 - c f \|_2^2=1+\|f\|^2-2|f(0)|,
\end{equation*}
which completes the proof.
\end{proof}

A similar inequality holding in $H^p$ would allow us to deduce that OPAs in $H^p$ cannot vanish in the disk. 

\begin{Proposition}\label{DI}
Let $1 < p < \infty$. Suppose $f \in H^p$ with $f(0) \neq 0$ and let $J$ be any non-constant inner function with $J(0)\neq 0$. If there exists a constant $c \in\mathbb{C}\setminus\{0\}$ such that 
\begin{equation*}
    \| 1 - J f \|_p \geq \|1 - cf\|_p, 
\end{equation*}
then any non-trivial optimal polynomial approximant in $H^p$ cannot vanish in the open unit disk.  
\end{Proposition}
\begin{proof}
Let $g \in H^p$ with $g(0) \neq 0$. Without loss of generality, assume that the degree of $q_{n,p}[g]$ is $n$. If $q_{n,p}[g]$ has zeros in $\D,$ we write $q_{n,p}[g] = B_k p_n,$ where $B_k$ is a Blaschke product of degree $k$ with $1 \leq k \leq n $ and $p_n$ is a polynomial of degree at most $n$ with no zeros in the disk. Then, by hypothesis,
\begin{equation*}
    \|1 - q_{n,p}[g]g \|_p = \|1 - B_kp_ng\|_p \geq \|1-cp_ng\|_p
\end{equation*}
for some constant $c\in\mathbb{C}$. However, by minimality of $q_{n,p}[g]$ and since $cp_n$ has degree at most $n$, we get $q_{n,p}[g] = cp_n$, which is a contradiction since $p_n$ has no zeros in the disk.
\end{proof}

In light of Proposition \ref{DI_p2}, we conjecture the following: 

\begin{Conjecture}
Let $1 < p < \infty$ and $f \in H^p$ with $f(0) \neq 0$. If $J$ is any non-constant inner function, then there exists a constant $c = c_{f}\in\mathbb{T}$ such that 
\[
\| 1 - J f \|_p > \|1 - cf\|_p.
\]
\end{Conjecture}

By continuity, we can at least say that the desired inequality still holds for values of $p$ which are close to two. In particular, using \cite[Lemma~3.1.1]{ChengFelder}, we obtain the following:
\begin{Proposition}\label{zero-free}
    Suppose $f \in H^{\infty}$ and $f(0) \neq 0$. Then there exists a neighborhood of $p=2$ such that the optimal polynomial approximants to $1/f$ in $H^p$ do not vanish in $\D$.
\end{Proposition}

\begin{Remark}
    We note that, using the continuity properties of the optimal polynomial approximants, we can replace $\mathbb{D}$ with $\overline{\mathbb{D}}$ in the previous result. Indeed, under the assumptions of Proposition \ref{zero-free}, we know that, for any $n\in\N$, the polynomials $q_{n,p}[f]$ converge to $q_{n,2}[f]$ as $p\to2$ uniformly on $\overline{\D}$ (see Lemma 3.1.1 in \cite{ChengFelder}). Noting that $q_{n,2}[f]$ does not vanish on $\overline{\D}$ establishes the claim.
\end{Remark}

\subsection{Behavior of OPAs and bounds on their roots}

 Using the results of the previous sections, we can say more about the roots of optimal polynomial approximants in $H^p$. First, we get that the roots always escape any compact subset of $\mathbb{D}$ as $n$ increases ad infinitum, which is an improvement of Proposition 5.1 in \cite{CCF}.
\begin{Proposition}\label{esc}
    Let $1 < p < \infty$ and $f \in H^p$ with $f(0) \neq 0$. Then, for any compact subset $K$ of $\mathbb{D}$, there exists $N\in\N$ such that the roots of $q_{n,p}[f]$ lie outside $K$ for all $n\geq N$.
\end{Proposition}
\begin{proof}
    Let $f = J F$, with $J$ inner and $F$ outer, and let $g_p^*$ be the metric projection of 1 on $[f]_p$ as given in Theorem \ref{s_and_s}. We know that $q_{n,p}[f]f\to g_p^*$ in $H^p$ (see for example Proposition 3.0.1 in \cite{ChengFelder}). Since the inner factor of $g_p^*$ is $J$, we also get that the sequence $q_{n,p}[f]F$ converges to the outer part of $g_p^*$ in $H^p$ which of course does not vanish in $\mathbb{D}$. The result follows by Hurwitz's Theorem.
\end{proof}
In \cite{Cent}, Centner gave the following lower bound: let $1 < p < \infty,f \in H^p$ with $f(0) \neq 0$ and $n\in\mathbb{N}$, then any root of $q_{n,p}[f]$ must lie outside the open disc of radius $(1-\|1-q_{n,p}[f]f\|_p^p)^{1/2}$ centered at the origin (see \cite[Proposition 5.1]{Cent}). We next provide an improvement.
\begin{Proposition}\label{opa_roots_prop}
Let $1 < p < \infty$ and $n\in\N$. Suppose that $f \in H^p$ has inner part $J$, with $J(0) \neq 0$. Let also $w_1,\dots,w_k$ be the roots (counting multiplicities) of $q_{n,p}[f]$ in $\mathbb{D}$. Then
\begin{equation}\label{lower_bound_zeros}
|w_1\cdots w_k| \geq \frac{(1 - \|1 - q_{n,p}[f]f\|_p^p)^{1/2}}{|J(0)|}.
\end{equation}
\end{Proposition}

\begin{proof}
We have that $q_{n,p}[f]=B_kp_n$ where $B_k$ is the Blaschke product formed by the roots $w_1,\dots,w_k$ and $p_n$ is a polynomial of degree at most $n$ that does not vanish in $\D$. Then, by Theorem \ref{fstarnozero_0}, we have
\begin{align*}
    \|1 - q_{n,p}[f]f\|_p^p
    &= \|1 - B_kp_nf\|_p^p \\
    &\geq\inf_{\varphi \in H^p} \left\| 1 - B_k J\varphi \right\|_p^p\\
    &= 1 - |w_1\cdots w_k J(0)|^2.
\end{align*}
Rearranging the terms of the inequality yields the claim.
\end{proof}
\begin{Remark}
    Note that since 
    \begin{equation*}
       \lim_{n \rightarrow \infty} \|1 - q_{n,p}[f]f\|_p^p = \operatorname{dist}_{H^p}^p(1, [f]_p) = 1 - |J(0)|^2,        
    \end{equation*}
   the right-hand side of \eqref{lower_bound_zeros} converges to 1 as $n$ increases to $\infty$. Thus, Proposition \ref{opa_roots_prop} can be viewed as a quantitative version of Proposition \ref{esc}.
\end{Remark}
Next, we give bounds for the roots of optimal polynomial approximants that depend on whether $p$ is greater or less than two.

\begin{Proposition}
Let $1 < p < 2$ and $n\in\mathbb{N}$. Suppose that $f \in H^p$ has inner part $J$, with $J(0) \neq 0$. Let also $w_1,\dots,w_k$ be the roots (counting multiplicities) of $q_{n,p}[f]$ in $\mathbb{D}$, then
\[
|w_1\cdots w_k| \geq \frac{\left(1 - \left[1 - \frac{|f(0)|^2}{\|f\|_2^2}\right]^{p/2} \right)^{1/2}}{|J(0)|}.
\]
In particular, if $f = J$ is inner, then
\[
|w_1\cdots w_k| \geq \frac{\left(1 - \left[1 - |J(0)|^2 \right]^{p/2} \right)^{1/2}}{|J(0)|}.
\]
\end{Proposition}

\begin{proof}
We combine Proposition \ref{opa_roots_prop} with the note following Proposition 4.0.15 in \cite{ChengFelder}, which says that
\[
\|1 - q_{n,p}[f]f\|_p \leq \left(1 - \frac{|f(0)|^2}{\|f\|_2^2} \right)^{1/2}.
\]
\end{proof}
For the case $2 < p < \infty$, we first need the following lemma: 
\begin{Lemma}\label{0_opa}
    Let $1<p<\infty$, and let $f \in H^p$.  Then
    \[
           \|1-q_{0,p}[f]f\|_p < 1
    \]
    if and only if $f(0) \neq 0$. In this case,
    \[
    \|1-q_{0,p}[f]f\|_p^r  \leq \frac{A^{r/(r-1)}}{(1 + A^{1/(r-1)})^r} + \frac{A}{(1 + A^{1/(r-1)})^r} < 1,
     \]
     where $K$ and $r$ are the relevant upper Pythagorean parameters and
     \[
         A:= K\bigg\| \frac{f-f(0)}{zf(0)}  \bigg\|^r_p
     \]
\end{Lemma}

\begin{proof}
There is no harm in first assuming $f(0) = 1$, so suppose $f = 1 + z\phi \in H^p$, with $\phi$ nonzero.  Then for $0<c\leq 1$, we have
\[
     \int_{\mathbb{T}} (1-c)^{\langle p-1\rangle} z\phi\, dm = 0.
\]
Consequently, $(1-c) \perp_p z\phi$ holds, and by Theorem \ref{pythagthm} we have
\begin{align}
    1 - cf(z) &=  (1 - c) +  (-cz\phi(z)) \nonumber \\
    \|1 - cf\|_p^r  &\leq (1-c)^r + K\|cz\phi\|_p^r  \nonumber \\
         &= (1-c)^r + Kc^r \|\phi\|_p^r, \label{upbdforerr}
\end{align}
where $r$ and $K$ are the applicable Pythagorean parameters.

Write $A := K\|\phi\|_p^r$. By elementary calculus, the expression
\[
       (1-c)^r + Ac^r 
\]
is critical when 
\begin{align*}
     0 &= r(1-c)^{r-1}(-1) + rAc^{r-1} \\ 
     (1-c)^{r-1} &= Ac^{r-1} \\
       c  &=  \frac{1}{1 + A^{1/(r-1)}}.
\end{align*}
Thus the expression \eqref{upbdforerr} takes the minimum value
\[
     \frac{A^{r/(r-1)}}{(1 + A^{1/(r-1)})^r} + \frac{A}{(1 + A^{1/(r-1)})^r},
\]
and hence it suffices to show this quantity is less than 1.

Indeed, it is elementary to see that for any $B>0$ we have
\begin{align*}
   B^r\Big( 1 + \frac{1}{B}\Big)  &<  B^r\Big( 1 + \frac{1}{B}\Big)^r \\
   B^{r-1} + B^r  &<  (1 + B)^r  \\
   \frac{B^{r-1} + B^r}{(1 + B)^r } &< 1.
\end{align*}
Substituting $B = A^{1/(r-1)}$, we conclude
\[
    \|1 - cf\|_p^r  \leq \frac{A^{r/(r-1)}}{(1 + A^{1/(r-1)})^r} + \frac{A}{(1 + A^{1/(r-1)})^r} < 1.
\]
More generally, for $f(0)\neq 0$, apply the above argument to $f(z)/f(0)$, and make the corresponding change to the definition of $A$.

Conversely, if $f(0)=0$, then the Mean Value Property for subharmonic functions immediately yields the minimal choice  $c=0$.

\end{proof}
\begin{Proposition}
Let $2 < p < \infty$ and $n\in\mathbb{N}$. Suppose that $f \in H^p$ has inner part $J$, with $J(0) \neq 0$. Let also $w_1,\dots,w_k$ be the roots (counting multiplicities) of $q_{n,p}[f]$ in $\mathbb{D}$, then
\[
|w_1\cdots w_k| \geq \frac{1}{|J(0)|} \left(1 - \left( \frac{(p - 1)\|f - f(0)\|_p^2}{|f(0)|^2 + (p - 1)\|f - f(0)\|_p^2} \right)^{p/2} \right)^{1/2}.
\]
\end{Proposition}

\begin{proof}
First note that 
\[
\|1 - q_{0,p}[f]f\|_p^p \geq \|1 - q_{n,p}[f]f\|_p^p \geq 1 - |w_1\cdots w_kJ(0)|^2,
\]
and apply the upper bound for $\|1 - q_{0,p}[f]f\|_p^p$ from Lemma \ref{0_opa}. For $2 < p < \infty$, the parameters are $r = 2$ and $K = p - 1$, and the bound simplifies to
\[
\|1 - q_{0,p}[f]f\|_p^{2} \leq \frac{A}{A + 1}, \quad \text{where } A = \frac{(p - 1)\|f - f(0)\|_p^2}{|f(0)|^2}.
\]
Straightforward computation now yields the desired result. 
\end{proof}

\section{Concluding Remarks and Open Questions}\label{concl}
We end with a few comments and questions. 

\begin{Remark}[The cases $0< p \le 1$ and $p=\infty$]
As is well known, the solution of extremal problems similar to problems considered in this paper need not be unique in $H^p$ for $0 < p \leq 1$ and $p = \infty$.  Nevertheless, it would be interesting to see what one can say about metric projections and OPAs in those spaces. Birkhoff-James orthogonality does not hold for $0 < p < 1,$ since the dual space is trivial. Along the same lines, it is natural to ask if the tools developed here can give a new unified proof of Beurling's Theorem for all $0 < p < \infty$, versus diverting the proof to the case $p = 2$ via factorization of $H^p$ functions as is done in \cite[p. 98]{MR628971}.
We also note that for $p = \infty$, as soon as any OPA vanishes in $\overline{\D}$, it must vanish identically by the Maximum Principle.
\end{Remark}

\begin{Remark}[Outer factor in the extremal function]
When $p = 2$, the projection of $1$ onto a nontrivial closed $z$-invariant subspace gives a constant times an inner function, while for $1 < p < \infty$ ($p \neq 2$), we get a non-constant outer function times an inner function. It would be interesting to gain a better understanding of why a non-constant outer function appears in the projection in the Banach space setting, versus the Hilbert space setting. 
\end{Remark}

\begin{Remark}[Universality]
The continuity of metric projections implies that, given any compact set $K$, the mapping $\mathcal{Q}_n \colon  H^p\setminus\{0\} \to C(K)$, which takes a function $f$ to its OPA $q_{n,p}[f]$, is continuous. This was a key ingredient for establishing the existence of functions in $H^2$ whose optimal polynomial approximants have \textit{universal} approximation properties on subsets of the unit circle with zero arclength measure (see \cite[Proposition~2.1]{BIMS}). The second main ingredient in proving the universality result is the following formula, which describes the explicit dependence of an OPA on its outer part (see \cite[Proposition~2.2]{BIMS}): 
If $g$ is an inner function in $H^2$ and $f\in H^2\setminus\{0\}$, then $q_{n,2}[fg]=\overline{g(0)} q_{n,2}[f]$. Thus, it is natural to examine the analogous result for the $H^p$ setting, which could provide results about universality in the $H^p$ setting.
\end{Remark}

\begin{Remark}[Dynamics of zeros]
It would also be interesting to study the interplay between the dynamics of the zeros of $H^p$ functions and the corresponding extremal problem \eqref{ext_prob}.  Moving a zero of a Blaschke product $J$ closer to the unit circle gives rise to a smaller distance between $1$ and the corresponding $z$-invariant subspace generated by $J$ (as seen from the distance formula in Theorem \ref{fstarnozero_0}).  For example, if  $J$ has zeros $a_1, a_2,\ldots$, then $\inf_{h \in H^p}\|1-Jh\|_p^p = 1 - |a_1 a_2 \cdots |^2$. So, if a zero $a_1$ wanders closer to the unit circle, the corresponding distance decreases. Along the same lines, it would be interesting to study the zeros of the first degree OPA as $p$ varies, for a fixed $f$. It is known, for instance, that if $f$ is a bounded analytic function, then the linear OPA for $1/f$ varies uniformly with $p$ \cite[Lemma 3.1.1]{ChengFelder};  moreover, if $f$ is inner, then the root of this linear OPA is bounded from the origin by an amount depending only on $p$ \cite[Theorem 5.1.3]{CCF}. Recall also that in $H^2$ and in certain other Dirichlet-type Hilbert spaces, as long as $1/f$ is not analytic in the closed disk, although the zeros of the OPAs stay outside the closed unit disk, a 
Jentzsch-type phenomenon occurs \cite{BKLSS}: that is, every point on the unit circle is a limit of the zeros of the OPAs of $1/f$. Does such a Jentzsch-type phenomenon occur in $H^p$ for $p \neq 2$?
\end{Remark}

\begin{Remark}[Bergman spaces]
Finally, the notion of $z$-invariant subspaces and corresponding extremal problems as considered in Sections 3 and 4 can easily be reformulated for Bergman spaces of analytic functions.  Recall that for $ 0 < p < \infty,$ the Bergman space $A^p$ is defined as
\begin{equation*}
    A^p := \left\{ f\in \operatorname{Hol}(\D) : 
    \|f\|_{A^p}^p := \int_{\D} |f(z)|^p\, dA(z) < \infty \right\},
\end{equation*}
where $dA$ is normalized area measure on the unit disk. For $1 \leq p < \infty,$ $A^p$ is a Banach space. The lattice of $z$-invariant subspaces is much more complicated in $A^p$ than in $H^p$, and elements of this lattice are not necessarily singly-generated (see \cite{MR808268, MR1241125}). Moreover, there is no factorization of $A^p$ functions as transparent as the factorization of functions in Hardy spaces. However, the notion of an inner function has been meaningfully extended to Bergman spaces: a function $G \in A^p$ is called $A^p$-inner if for $n = 0, 1, 2, \ldots$, we have 
\begin{equation*}
    \int_{\D} \left(|G(z)|^{p} -1 \right) z^ n \, dA(z) = 0.
\end{equation*}
Thus, if $G$ is an $A^p$-inner function and we consider the $z$-invariant subspace $M = [G]_{A^p}$, one can investigate metric projections of $1$ onto $M$ in a way that is analogous to investigating metric projections of $1$ onto $z$-invariant subspaces in $H^p$. Following results in \cite{MR1133317} for $p = 2$, the orthogonal projection of $1$ onto $[G]_{A^2}$ is given by $\overline{G(0)}\, G$, as is true within the analogous setup in $H^2$. Therefore, it is natural to guess that for $1 < p < \infty$ ($p \neq 2$), the metric projection of $1$ onto $[G]_{A^p}$ gives a cyclic vector times $G$, in analogue with the conclusion of Theorem \ref{s_and_s} (here, a function $f\in A^p$ being cyclic means that $[f]_{A^p} = A^p$). 
\end{Remark}

\subsection*{Acknowledgments} This project was partially conducted during visits by C.\ B\'en\'eteau and D.\ Khavinson to University College Dublin and Indiana University. Travel to UCD was supported by the Distinguished Visitor Funding Scheme and M. Manolaki's Ad Astra Starting Grant. C.\ B\'en\'eteau, C.\, Felder, and D.\ Khavinson would like to thank the Indiana University Department of Mathematics for hospitality during a subsequent visit. D.\ Khavinson acknowledges support from Simons Foundation grant 513381 and K.\ Maronikolakis acknowledges support from the Irish Research Council.

\bibliographystyle{plain}
\bibliography{HpBibliography.bib}

\end{document}